\documentclass[12pt]{article}
\usepackage{graphicx}
\usepackage{color}
\usepackage{amsmath}
\usepackage{amsthm}
\newtheorem{statement}{statement}[section]

\newtheorem{corollary}[statement]{Corollary}

\newtheorem{lemma}[statement]{Lemma}
\newtheorem{proposition}[statement]{Propositio
n}

\newtheorem{remark}[statement]{Remark}
\newtheorem{theorem}[statement]{Theorem}
\catcode`\Ž=\active \def Ž{\'e}
\catcode`\=\active \def {\`e}
\catcode`\ˆ=\active \def ˆ{\`a}
\catcode`\=\active \def {\`u}
\catcode`\"=\active \def "{\`\i}
\catcode`\˜=\active \def ˜{\`o}
\catcode`\'=\active \def '{\"e}
\catcode`\Š=\active \def Š{\"a}
\catcode`\•=\active \def •{\"\i}
\catcode`\Ÿ=\active \def Ÿ{\"u}
\catcode`\š=\active \def š{\"o}
\catcode`\=\active \def {\^e}
\catcode`\‰=\active \def ‰{\^a}
\catcode`\"=\active \def "{\^\i}
\catcode`\™=\active \def ™{\^o}
\catcode`\ž=\active \def ž{\^u}
\catcode`\=\active \def {\c c}
\catcode`\'=\active \def '{\c C}
\catcode`\Ï=\active \def Ï{\oe}
\font\sevenrm=cmr10 scaled 700

\font\ninerm=cmr10 scaled 900
\font\tenrm=cmr10 scaled 1000
\font\fourteenbf=cmbx10 scaled 1400

\font\sectionbf=cmbx10 scaled 1700

\def\smallskip{\vspace{5pt}}
\def\medskip{\vspace{10pt}}
\def\Act{\hbox{\rm Act\ \!}}
\def\Ext{\hbox{\rm Ext\ \!}}
\def\Int{\hbox{\rm Int\ \!}}

\definecolor{white}{rgb}{1,1,1}

\definecolor{magenta}{rgb}{1,0.6,1}
\definecolor{yellow}{rgb}{1,1,0}
\definecolor{green}{rgb}{0.6,1,0.8}
\definecolor{red}{rgb}{1,0.4,0.6}
\definecolor{blue}{rgb}{0.8,1,1} 
\definecolor{tangerine}{rgb}{1,0.8,0.2} 

\def\m{\colorbox{magenta}}
\def\y{\colorbox{yellow}}
\def\g{\colorbox{green}}
\def\b{\colorbox{blue}}
\def\r{\colorbox{red}}
\def\t{\colorbox{tangerine}}

\def\square{\hbox{\vrule\vbox{\hrule\phantom{Z}\hrule}\vrule}}
\begin{document}

\title{\vspace{-3cm}
\textbf{The Tutte Polynomial\\ of a Morphism of Matroids\\}
{\fourteenbf 5. Derivatives as Generating Functions\\}
\vspace{-0.25cm}
{\fourteenbf of Tutte Activities}}

\author{Michel Las Vergnas$^{\scriptscriptstyle\diamondsuit}$}
\date{\ninerm April 10, 2012$^\star$}

\maketitle

\centerline{\it Dedicated to the memory of Yahya Ould Hamidoune}

\vspace{2cm}

\begin{abstract}
\tenrm
\baselineskip 12pt
We show that in an ordered matroid the partial derivative
\scalebox{0.87}{$\partial^{p+q}t/\partial x^p\partial y^q$} 
of the Tutte polynomial
is $p!q!$ times the generating function of activities
of subsets with corank $p$ and nullity $q$.
More generally, this property holds for the 3-variable Tutte polynomial of a matroid perspective.

\end{abstract}

\vfill
{\tenrm
\baselineskip 12pt
\vskip 0.1cm
\noindent
AMS Classification: Primary: 05B35\\  

\noindent
Keywords : matroid, matroid strong map, matroid perspective, independent, basis, spanning,
Tutte polynomial, Tutte polynomial derivative,
internal activity, external activity, corank, nullity, rank codrop,
Dawson partition, lattice complementation\\}

\bigskip
\hrule height .5pt depth 0pt width 6truecm
\medskip

{\ninerm
\noindent $^{\scriptscriptstyle\diamondsuit}$ C.N.R.S., Paris\par
\noindent $^\star$ European J. Combinatorics, to appear}

\eject

\noindent
{\sectionbf Introduction}

\medskip\noindent
Let $M$ be a matroid on a set $E$.
The Tutte polynomial $t(M;x,y)$ of $M$ can be defined by the
closed formula

\medskip\noindent
\hskip 1cm
$\displaystyle
t(M;x,y)=\sum_{A\subseteq E}(x-1)^{r(M)-r_M(A)}(y-1)^{|A|-r_M(A)}$
\hfill (1)

\medskip\noindent
where $r_M(A)$ denotes the rank of $A$ in $M$.

\bigskip
It is well-known that the Tutte polynomial can also be expressed as a generating function
in terms of Tutte activities - internal and external - of bases, 
providing a state model with numerous applications.

Suppose the set of elements $E$ of $M$ is linearly ordered.
Let $\iota_M(B)$ and $\epsilon_M(B)$ denote the internal and external
activities of a basis $B$ of $M$.
We have

\medskip\noindent
\hskip 1cm
$\displaystyle
t(M;x,y)=\sum_{B\subseteq E{\hbox{\sevenrm\ basis of }}M}x^{\iota_M(B)}y^{\epsilon_M(B)}$
\hfill (2)

\medskip
This formula first appeared in a  founding paper of W.T. Tutte on graphs
in 1954 \cite{Tu54}.
It has been extended to matroids by H.H. Crapo in 1969 \cite{Cr69}. 
The reader is referred to the 1992 textbook chapter by T. Brylawski and J. Oxley \cite{BO92} for an extensive survey on Tutte polynomials.
A recent survey in the case of graphs can be found in a 2011 book chapter by J.A. Ellis-Monaghan and C. Merino \cite{EMM11}.

\medskip
By differentiating (2), we get an expression for
${{\partial^{p+q}t}/{{\partial x^p}{\partial y^{q}}}}$
in terms of activities of bases.
However we no longer have a generating function.
Our purpose in the present note is to give an alternate expression of partial derivatives of the Tutte polynomial as a generating function of Tutte activities directly generalizing (2).

\medskip
The key tool is the Dawson partition of the Boolean lattice associated with a matroid or a matroid perspective.

\medskip
We show in Section 1 that a certain 4-variable identity involving the Tutte polynomial and the generating function of corank, nullity, internal and external activities of subsets due to G. Gordon and L. Traldi \cite{GT90} follows immediately from the expansion of partial derivatives as generating function.
From this 4-variable expression can be derived numerous 2-variable expansions of the Tutte polynomial.
Most of them are due to G. Gordon and L. Traldi; however several new ones are exhibited here.

\bigskip
We have introduced in 1975 \cite{LV75} the 3-variable {\it Tutte polynomial of a  matroid perspective} (definition in Section 2), and studied its properties in a series of papers: fundamental properties in \cite{LV80} \cite{LV99}, Eulerian partitions of 4-valent graphs imbedded in surfaces \cite{LV81}, activities of orientations in \cite{LV84}\cite{LV12}, vectorial matroids in \cite{ELV04}, computational complexity in \cite{LV07}.
The Tutte polynomial of a matroid perspective may equivalently be considered as associated with a ported matroid, i.e. a matroid with a distinguished set of elements \cite{LV75}\cite{LV80}\cite{LV99}.
Applications of the 3-variable Tutte polynomial of a ported matroid to electrical networks are discussed by S. Chaiken in \cite{Ch89}.

An equivalent form of the 3-variable Tutte polynomial has been consi\-dered in 2004 as a 4-variable polynomial, called the {\it linking polynomial}, by D.J.A. Welsh and K.K. Kayibi \cite{WK04}\cite{WK05}.
Under this form, K.K. Kayibi has established in \cite{Ka08} a gene\-ralization to an ordered matroid perspective of the decomposition introduced in \cite{ELV98} for the Tutte activities of an ordered matroid.

\medskip
In Section 2, we show that a partial derivative of the Tutte polynomial of a matroid perspective can also be expressed as a generating function of Tutte activities.

The 4-variable identity for the Tutte polynomial of a matroid generalizes naturally into a 5-variable identity involving the 3-variable Tutte polynomial of a matroid perspective $M\rightarrow M'$ and a generating function of corank and internal activity of subsets in $M'$, and  nullity and external activity of subsets in $M$, and the difference of ranks of a subset in $M$ and in $M'$. 
From this identity follows easily many 3-variable expansions as generating function of the 3-variable Tutte polynomials.

\medskip
These expansions can be classified into equivalence classes, in such a way that expansions in a class contain the same summands.
In Section 3, extending previous works of J.E. Dawson \cite{Da81}, and of G. Gordon and L. Traldi \cite{GT90}, we show how this equivalence follows from dualities between subsets, associated with corank, nullity, internal and external activity parameters.

\bigskip
Finally, we mention a closely related paper dealing with orientation activities in ordered oriented matroids \cite{LV12}.

\eject
\noindent
{\sectionbf Dawson Partitions}
	
\medskip\noindent
For the convenience of the reader, we state in this section the main properties  of Dawson partitions of the Boolean lattice.

\medskip
J.E. Dawson has introduced in 1981 \cite{Da81} the following construction.
Let $E$ be a (finite) linearly ordered set,
and ${\cal P}\subseteq 2^E$ be any set of subsets of $E$.
For $A\subseteq E$,  let $f(A)\in\cal P$ be the (unique) set 
$X\in\cal P$ such that the symmetric difference $A\Delta X$ is smallest for the colexicographic ordering of subsets in $\{A\Delta X\mid X\in\cal P$\}.

\medskip
For any $X\in\cal P$, the inverse image $f^{-1}(X)$ is an interval 
$[g(X),h(X)]$ of $2^E$ for the inclusion ordering.
We have clearly $X=f(X)$, hence $X\in [g(X),h(X)]$.
The intervals $f^{-1}(X)$ for $X\in\cal P$ are pairwise disjoint, and constitute the {\it Dawson partition of $2^E$ defined by $\cal P$}.

\medskip
A nice converse has been recently established by J. Brunat, A. Guedes de Oliveira and M. Noy \cite{BGN09}: a partition of $2^E$ into intervals $[X_1,Y_1],[X_2,Y_2],\ldots,\break [X_p,Y_p]$, with indices such that $X_1<X_2<\ldots<X_p$, is a Dawson partition if and only if $Y_1<Y_2<\ldots<Y_p$.

\medskip
The paper of Dawson contains several further results describing $f$ and associated operators.
In the present paper, we are mainly interested in the cases when $\cal P$ is the set of bases of a matroid, or, more generally, the set of independent/spanning sets of a matroid perspective.

The matroid case is studied in details by Dawson on pages 143--147 of \cite{Da81}.

\medskip
Let $M$ be a matroid on a linearly ordered set $E$, and $A\subseteq E$.
Set

\medskip\noindent
$P_M(A):=\{e\in E\setminus A \mid$\hfill\\
\null\hfill $\hbox{\it $e$ is the smallest element of some cocircuit of $M$ contained in $E\setminus A$}\}$

\smallskip\noindent
$Q_M(A):=\{e\in A \mid$\hfill\\
\null\hfill $\hbox{\it $e$ is the smallest element of some circuit of $M$ contained in $A$}\}$

\medskip
By matroid duality $$Q_M(A)=P_{M^*}(E\setminus A)$$

\medskip
It can be easily shown that 
$$|P_M(A)|=r(M)-r_M(A)$$
and
$$|Q_M(A)|=|A|-r_M(A)$$
\indent
The parameters $cr_M(A)=|P_M(A)|$ and $nl_M(A)=|Q_M(A)|$ are respectively called the {\it corank} and the {\it nullity} of $A$.

\medskip
Set $$f_M(A):=A\cup P_M(A)\setminus Q_M(A),$$
and let $\cal P$ be the set of bases of $M$.
Then $$B:=f_M(A)\in{\cal P}$$ is a basis of $M$, and the Dawson interval 
containing $A$ is $[g_M(A),h_M(A)]$, where
$$g_M(A):=g_M(B)=B\setminus\Int_M(B)$$
$$h_M(A):=h_M(B))=B\cup\Ext_M(B)$$
where, $\Int_M(B)$ resp. $\Ext_M(B)$ denotes the set of internally resp. externally active elements in the sense of Tutte-Crapo with respect to $B$ in $M$.

\bigskip
Most results of Dawson for matroids, in particular the partition of the Boolean lattice into intervals associated with bases, already appear in a 1969 paper by H. Crapo \cite{Cr69} Proposition 12, and can also be found in a 1990 paper by G. Gordon and L. Traldi \cite{GT90}.

\bigskip
A second case of interest here, generalizing the matroid case, is that of a matroid perspective (definitions are recalled in Section 2). 
Matroid perspectives are not considered by any of the former authors, H.H. Crapo, J.E. Dawson (however his general theory applies), or G. Gordon and L. Traldi.

\medskip
We will establish directly in Section 2 the specific properties of Dawson partitions of matroid perspectives needed in the proof of the main theorem.

\bigskip\bigskip
\section{Derivatives of Matroid Tutte Polynomials}

In full generality, the results of the paper hold for 3-variable Tutte polynomials of matroid perspectives.
However, many readers, whereas at ease with matroids, may be not so familiar with matroid perspectives and their Tutte polynomials. 
For their convenience, we present in this section, without proofs, the main results in terms of Tutte polynomials of matroids.
The general case of matroid perspectives, with proofs, will be considered in Section 2.

\medskip
We first recall the classical definitions of internal and
external activities of matroid bases.

\medskip
Let $M$ be a matroid on a linearly ordered set $E$, and $B$ be a basis of $M$.

A nonbasis element $e\in (E\setminus B)$ is {\it externally active} with respect to $B$ if $e$ is the smallest element of the {\it fundamental circuit} $C_M(E)$, i.e. of the unique circuit contained in $B\cup\{e\}$.
The {\it external activity} $\epsilon_M(B)$ is the number of externally active elements of $E\setminus B$.

A basis element $e\in B$ is {\it internally active} with respect to $B$
if $e$ is the smallest element of the {\it fundamental cocircuit} $C^*_M(E)$, i.e. of the unique cocircuit contained in $(E\setminus B)\cup\{e\}$.
The {\it internal activity} $\iota_M(B)$ is the number of internally active elements of $B$.

\bigskip
Extensions of these definitions to any subset are straightforward
\cite{Da81}\cite{GT90}\cite{LV80}.

\medskip
Let $A\subseteq E$.
Set

\medskip\noindent
$\Ext_M(A):=e\in E\setminus A \mid$\hfill\\
\null\hfill $e\ \hbox{\it is the smallest element of some circuit of $M$ contained in } A\cup\{e\}\}$,\par

\medskip\noindent
and let
$$\epsilon_M(A)=|\Ext_M(A)|.$$
\indent
We say that $\Ext_M(A)$ is the set of {\it externally active} elements of $A$ with respect to $M$.

\medskip
Dually, we set 
$$\Int_M(A)=\Ext_{M^*}(E\setminus A),$$ 
and
$$\iota_M(A)=\epsilon^*(E\setminus A).$$
\indent
We have

\medskip\noindent
$\Int_M(A):=\{e\in A \mid$\hfill\\
\null\hfill $ e\ \hbox{\it is the smallest element of some cocircuit of $M$ contained in } (E\setminus A)\cup\{e\}\}$.

\medskip
We say that $e\in \Int_M(A)$ is the set of {\it internally active} elements of $A$ with respect to $M$.

\medskip
When applied to a basis, these definitions reduce to the usual ones.

\bigskip
Despite the different notation, we observe that $\Int(A)$ and $P(A)$ on one hand, and $\Ext(A)$ and $Q(A)$ on the other, are very close.
Set 

\medskip\noindent
$\Act_M(A):=\{e\in E\mid$\hfill\\
\null\hfill $e\ \hbox{\it is the smallest element of some circuit of $M$ contained in } A\cup\{e\} \}$,

\medskip
and $$\Act^*_M(A):=\Act_{M^*}(A).$$
Then 
$$\iota_M(A)=\Act^*_M(E\setminus A)\cap A,$$ 
$$P_M(A)=\Act^*_M(E\setminus A)\setminus A,$$ 
$$\epsilon_M(A)=\Act_M(A)\setminus A,$$ 
$$Q_M(A)=\Act_M(A)\cap A.$$ 

\bigskip
The main difference for activities in the general case as compared to the basis case is that the (co)circuit in the definition may not be unique.
However, uniqueness is easily recovered as follows.

\begin{proposition}
\label{prop:uniqueness}
Let $M$ be a matroid on a linearly ordered set $E$, and $A\subseteq E$.
Then for $e\in\Ext_M(A)$ there is a unique circuit $C$ of $M$ with smallest element $e$ such that $e\in C\subseteq (A\setminus Q_M(A))\cup\{e\}$.
\end{proposition}

\begin{proof}
Let $e\in\Ext(A)$.
There is a circuit $C$ with smallest element $e$ such that $e\in C\subseteq A\cup\{e\}$.
Let $C$ be chosen such that the smallest element of $C\cap A$ is the greatest possible.
Suppose there exists $x\in C\cap Q(A)$.
By definition of $Q(A)$, there is a circuit $X\subseteq A$ with smallest element $x$.
By elimination of $x$ from $C$ and $X$, there is a circuit $C'$ such that $e\in C'\subseteq (C\cup X)\setminus\{x\}\subseteq A\cup\{e\}$.
Then the minimality of $e$ in $C$ and that of $x$ in $X$ readily imply that $e$ is the smallest element of $C'$, and that the smallest element of $C'\cap A$ is $>x$, contradicting the definition of $C$.
Hence $C\cap Q(A)=\emptyset$.

The uniqueness of a circuit $C$ such that $e\in C\subseteq (A\setminus Q_M(A))\cup\{e\}$ is immediate.
If there were two such circuits, then, eliminating $e$, there would exist a circuit contained in $A\setminus Q_M(A)$, contradicting the definition of $Q$.
\end{proof}

With notation of the previous section, the Dawson interval 
$[g_M(A),h_M(A)]$ is given by
$$g_M(A)=g_M(B)=B\setminus\Int_M(A)=A\setminus\Int_M(A)\setminus Q_M(A),$$ 
and
$$h_M(A)=h_M(B)=B\cup\Ext_M(A)=A\cup\Ext_M(A)\cup P_M(A).$$ 

\bigskip
The main result of this section is the following.

\begin{theorem}
\label{thm:matroid_partial_derivative}
\noindent
Let $M$ be a matroid on a linearly ordered set $E$, and $p$, $q$
be non negative integers.
Then
$${{\partial^{p+q} t}\over{{\partial x^p}{\partial y^q}}}
(M;x,y)=
p!q!\sum_{\buildrel {A\subseteq E}\over
{\buildrel{\scriptstyle cr_M(A)=p}\over {\scriptstyle nl_M(A)=q}}}x^{\iota_M(A)}y^{\epsilon_M(A)}.$$
\end{theorem}

\bigskip
Theorem \ref{thm:matroid_partial_derivative} follows from properties of the Dawson partition of the Boolean lattice defined by matroid bases.
It will proved in the next section in the more general setting of matroid perspectives, see 
Theorem \ref{thm:perspective_derivative}.

We point out that by Theorem \ref{thm:matroid_partial_derivative} non zero partial derivatives are associated in a very simple way with a partition of the Boolean lattice: the partial derivative ${\partial^{p+q} t}/{{\partial x^p}{\partial y^q}}$ corresponds to subsets of corank $p$ and nullity $q$.

\medskip
As a corollary to Theorem \ref{thm:matroid_partial_derivative}, by the dualities used in \cite{GT90} Example 3.1 (see below Theorem \ref{thm:dualities}), we get the following alternate expansions of partial derivatives.

\medskip
\begin{corollary}
\label{cor:m_alternate}
We have
$${{\partial^{p+q} t}\over{{\partial x^p}{\partial y^q}}}                                         
(M;x,y)
=p!q!\sum_{\buildrel {A\subseteq E}\over
{\buildrel\scriptstyle {\iota_M(A)=p}\over{\scriptstyle nl_M(A)=q}}}
x^{cr_M(A)}y^{\epsilon_M(A)},$$
$${{\partial^{p+q} t}\over{{\partial x^p}{\partial y^q}}}                                         
(M;x,y)
=p!q!\sum_{\buildrel {A\subseteq E}\over
{\buildrel\scriptstyle {cr_M(A)=p}\over{\scriptstyle \epsilon_M(A)=q}}}
x^{\iota_M(A)}y^{nl_M(A)},$$

and

\hspace{2.1cm}
$\displaystyle {{\partial^{p+q} t}\over{{\partial x^p}{\partial y^q}}}                                         
(M;x,y)
=p!q!\sum_{\buildrel {A\subseteq E}\over
{\buildrel\scriptstyle {\iota_M(A)=p}\over{\scriptstyle \epsilon_M(A)=q}}}
x^{cr_M(A)}y^{nl_M(A)}.$
\hfill\square
\end{corollary}

\medskip\noindent
{\bf Example 1}

\begin{figure}[h] 
\centerline{\includegraphics[scale=0.35]{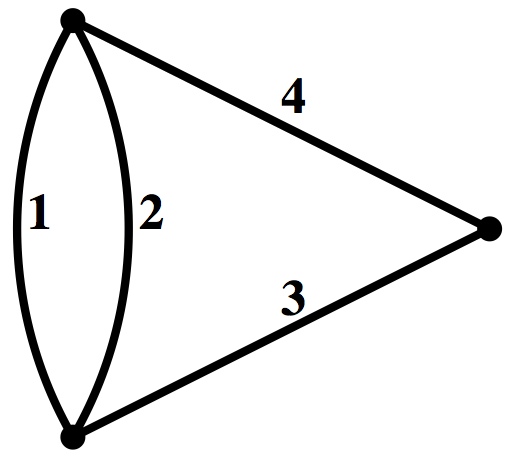}}
\end{figure}
The Tutte polynomial of the graphic matroid defined by the cycle space of  the above graph is $t(x,y)=x^2+xy+y^2+x+y$.
We have

\bigskip
\centerline{
\scalebox{0.9}{
\(
\begin{array}{|c||c|c|c|c|c|}
\hline
\hbox{spanning tree}&
\begin{array}{c} \phantom{} \\ \phantom{} \end{array} 13\begin{array}{c} \phantom{} \\ \phantom{} \end{array} &23&34&14&24 \\ \hline
\begin{array}{c}\hbox{intern. active edge(s): left of $\star$}\\
\hbox{extern. act. edge(s): right of $\star$}\end{array}
& 13\ *&3\ *\ 1&*\ 12& 1\ * & *\ 1 
\begin{array}{c}\phantom{} \\ \phantom{} \end{array} \\ \hline
\hbox{Dawson interval}& [\emptyset,13] & [2,123] & [34,1234] & [4,14] & [24,124] 
\begin{array}{c}\phantom{} \\ \phantom{} \end{array}\\ \hline\hline
t(x,y)=x^2+xy+y^2+x+y&
\begin{array}{c} 13 \\ x^2 \end{array}&
\begin{array}{c} 23 \\ xy \end{array}&
\begin{array}{c} 34 \\ y^2 \end{array}&
\begin{array}{c} 14 \\ x \end{array}&
\begin{array}{c} 24 \\ y \end{array}\\ \hline
{{\partial t}\over{\partial x}}(x,y)=2x+y+1&
\begin{array}{cc} 
  \begin{array}{c} 1 \\ x \end{array}&
  \begin{array}{c} 3 \\ x \end{array} 
\end{array}&
\begin{array}{c}  2 \\ y \end{array} &
\begin{array}{c} \\ \end{array} &
\begin{array}{c} 4 \\ 1 \end{array} &
\begin{array}{c} \\ \end{array}\\ \hline
{{\partial t}\over{\partial y}}(x,y)=x+2y+1&&
\begin{array}{c} 123 \\ x \end{array}&
\begin{array}{cc} 
  \begin{array}{c} 134 \\ y \end{array}&
  \begin{array}{c} 234 \\ y \end{array} 
\end{array}&
\begin{array}{c} \\ \end{array}&
\begin{array}{c} 124 \\ 1 \end{array}\\ \hline
{1\over 2}{{\partial^2 t}\over{\partial x^2}}(x,y)=1&
\begin{array}{c} \emptyset \\ 1 \end{array}&
\begin{array}{c} \\ \end{array}&
\begin{array}{c} \\ \end{array}&
\begin{array}{c} \\ \end{array}&
\begin{array}{c} \\ \end{array}\\ \hline
{{\partial^2 t}\over{\partial x\partial y}}(x,y)=1&
\begin{array}{c} \\ \end{array}&
\begin{array}{c} 12 \\ 1 \end{array}&
\begin{array}{c} \\ \end{array}&
\begin{array}{c} \\ \end{array}&
\begin{array}{c} \\ \end{array}\\ \hline
{1\over 2}{{\partial^2 t}\over{\partial x^2}}(x,y)=1&
\begin{array}{c} \\ \end{array}&
\begin{array}{c} \\ \end{array}&
\begin{array}{c} 1234 \\ 1 \end{array}&
\begin{array}{c} \\ \end{array}&
\begin{array}{c} \\ \end{array}\\ \hline
\end{array}
\)
}}
\smallskip
\centerline{Table 1}

\medskip
\begin{corollary}
\label{cor:matroid_derivative}
\noindent
Let $M$ be a matroid on a linearly ordered set $E$, and $p$
be a non negative integer.
Then
$${{d^p t}\over{dx^p}}(M;x,x)=
p!\sum_{\buildrel {A\subseteq E}\over
{\scriptstyle cr_M(A)+nl_M(A)=p}}x^{\iota_M(A)+\epsilon_M(A)}.$$
\end{corollary}

\medskip
For a proof, see below Corollary \ref{cor:perspective_(x,x)_derivative}.

\begin{theorem}{\rm (equivalent to \cite{GT90} Theorem 3)}
\label{thm:4_variables}
Let $M$ be a matroid on a linearly ordered set $E$.
We have
$$t(M,M';x+u,y+v)=\sum_{A\subseteq E}x^{cr_M(A)}u^{\iota_M(A)}y^{nl_M(A)}v^{\epsilon_M(A)}.$$
\end{theorem}

\medskip\noindent
{\bf Example 1} (continued)

\medskip

\centerline{
\scalebox{1}{
\(
\begin{array}{|l||l|l|l|l||l|}
 \hline
\scalebox{0.75}{$A$} & 
\scalebox{0.75}{$\Int(A)$} & 
\scalebox{0.75}{$P(A)$} & 
\scalebox{0.75}{$\Ext(A)$} & 
\scalebox{0.75}{$Q(A)$} & 
\scalebox{0.75}{$x^{\iota}u^{cr}y^{\epsilon}v^{nl}$}\\ \hline\hline
13   & 13 &    &    &    & \r{$x^2$} \\ \hline
23   & 3  &    & 1  &    & \y{$xy$}  \\ \hline
34   &    &    & 12 &    & \g{$y^2$} \\ \hline
14   & 1  &    &    &    & x         \\ \hline
24   &    &    & 1  &    & y         \\ \hline
1    & 1  & 3  &    &    & \r{$xu$}  \\ \hline
3    & 3  & 1  &    &    & \r{$xu$}  \\ \hline
2    &    & 3  & 1  &    & \y{$uy$}  \\ \hline
4    &    & 1  &    &    & u         \\ \hline
123  & 3  &    &    & 1  & \y{$xv$}  \\ \hline
134  &    &    & 2  & 1  & \g{$yv$}  \\ \hline
234  &    &    & 1  & 2  & \g{$yv$}  \\ \hline
124  &    &    &    & 1  & v         \\ \hline
\emptyset & & 13 &  &    & \r{$u^2$} \\ \hline
12   &    & 3  &    &  1 & \y{$uv$}  \\ \hline
1234 &    &    &    & 12 & \g{$v^2$} \\ \hline
\end{array}
\)
}}

\medskip
\scalebox{0.8}{
The last column sums up to
\r{$(x+u)^2$}+\y{$(x+u)(y+v)$}+\g{$(y+v)^2$}$+(x+u)+(y+v)$,}

\vspace{-3pt}
\scalebox{0.8}{
in accordance with Theorem \ref{thm:4_variables}.}

\smallskip
\centerline{Table 2}

\medskip
For a proof of Theorem \ref{thm:4_variables}, see below Theorem \ref{thm:5_variables}.

\bigskip
Numerous 2-variable expansions of $t(M;x,y)$ follow from Theorem \ref{thm:4_variables} by specializing variables.
The most remarkable are obtained by setting some of $x$, $u$, $y$, $v$ to either 0 or 1, and/or replacing by $x/2$ $y/2$, and performing an appropriate change of variables.

We observe that setting a variable to 0 means that the corresponding terms is 0, except when the parameter in exponent is also 0, i.e. the sum is 
reduced to the class of subsets having this parameter equal to 0.
Setting a variable to 1 makes the corresponding term also equal to 1, so it disappears from the expansion.

\medskip
In Examples 3.1-3.5 of \cite{GT90}, G. Gordon and L. Traldi derive from Theorem 3 a total of 17=4+4+1+4+4 different remarkable expansions of the Tutte polynomial.
We will review rapidly these expansions, which will be generalized in Section 2 to 3-variable Tutte polynomials of matroid perspectives.

In \cite{GT90}, these expansions, are classified in 5 families, numbered here (1)-(5).
G. Gordon and L. Traldi have observed that expansions consist of the same summands within each of the 5 five families, but occurring in different orders.
In Section 3, we will discuss certain dualities, explaining this property.

We complete the Gordon-Traldi list, by exhibiting 8 new expansions in Example 3.3, classified in 4 families numbered (3b)-(3e).
The number of different expansions grows from 17 to 25, and the number of expansions with different summands from 5 to 9.

We denote by $[[x,u,y,v]]$ the expansion resulting from setting in Theorem \ref{thm:4_variables} the 4 variables $x,u,y,v$ in this order to the values displayed between the brackets.

We display the first expansion in each of the 9 essentially different families.

An example is shown in Table 3 below.

\medskip\noindent
$\bullet$ \cite{GT90} Example 3.1\\ 
$[[x-1,1,y-1,1]]$, $[[x-1,1,1,y-1]]$, $[[1,x-1,y-1,1]]$, $[[1,x-1,1,y-1]]$:

\smallskip
$\displaystyle t(M;x,y)=
\sum_{A\subseteq E}
(x-1)^{cr_M(A)}(y-1)^{nl_M(A)}$.
\hfill (1)

\smallskip
(1) is the usual definition of a Tutte polynomial in terms of cardinality and rank of subsets.

\eject
\bigskip\noindent
{\bf Example 1} (continued)

\medskip
\centerline{
\scalebox{0.75}{
\(
\begin{array}{|l||l|l|l|l||l|l|l|l|l|l|l|l|l|}
\hline
A & cr(A) & \iota(A) & nl(A) & \epsilon(A) & (1) & (2) & (3) & (3b) & (3c) & (3d) & (3e) & (4) & (5) \\ \hline\hline
\emptyset 
     & 2 & 0 & 0 & 0 & (x-1)^2  &    & x^2/4&    x^2/4& x^2/4&  (x-1)^2&  x^2 &        & (x-1)^2\\ 
1    & 1 & 1 & 0 & 0 &   x-1    &    & x^2/4&    x^2/4& x^2/4&     x-1 &      &        &  x-1   \\
2    & 1 & 0 & 0 & 1 &   x-1    &    &  xy/4&      x/2&      & (x-1)y/2&  xy/2&        & (x-1)y \\  
3    & 1 & 1 & 0 & 0 &   x-1    &    & x^2/4&    x^2/4& x^2/4&     x-1 &      &        &  x-1   \\  
4    & 1 & 0 & 0 & 0 &   x-1    &    &   x/2&      x/2&   x/2&     x-1 &   x  &        &  x-1   \\  
12   & 1 & 0 & 1 & 0 &(x-1)(y-1)&    &  xy/4& x(y-1)/2&  xy/2& (x-1)y/2&  xy/2&        &        \\  
13   & 0 & 2 & 0 & 0 &    1     & x^2& x^2/4&    x^2/4& x^2/4&       1 &      &    x^2 &    1   \\  
14   & 0 & 1 & 0 & 0 &    1     & x  &   x/2&      x/2&   x/2&       1 &      &      x &    1   \\  
23   & 0 & 1 & 0 & 1 &    1     & xy &  xy/4&      x/2&      &      y/2&      &      x &    y   \\  
24   & 0 & 0 & 0 & 1 &    1     & y  &   y/2&      1  &      &      y/2&   y/2&      1 &    y   \\  
34   & 0 & 0 & 0 & 2 &    1     & y^2& y^2/4&      1  &      &    y^2/4& y^2/4&      1 &   y^2  \\  
123  & 0 & 1 & 1 & 0 &   y-1    &    &  xy/2& x(y-1)/2&  xy/2&      y/2&      &  x(y-1)&        \\  
124  & 0 & 0 & 1 & 0 &   y-1    &    &   y/2&    y-1  &   y  &      y/2&   y/2&    y-1 &        \\  
134  & 0 & 0 & 1 & 1 &   y-1    &    & y^2/4&    y-1  &      &    y^2/4&  y^2/4&   y-1 &        \\  
234  & 0 & 0 & 1 & 1 &   y-1    &    & y^2/4&    y-1  &  y^2 &    y^2/4&  y^2/4&   y-1 &        \\  
1234 & 0 & 0 & 2 & 0 & (y-1)^2  &    & y^2/4& (y-1)^2 &      &    y^2/4&  y^2/4&(y-1)^2&
\\ \hline
\end{array}
\)
}} 

\medskip
\scalebox{0.8}{
\indent
Each of the 9 rightmost columns adds up to the Tutte polynomial of the graphic}

\vspace{-3pt}
\scalebox{0.8}{
matroid of Example 1, namely $x^2+xy+y^2+x+y$.}

\smallskip
\centerline{Table 3}

\bigskip\noindent
$\bullet$ \cite{GT90} Example 3.2\\ 
$[[0,x,0,y]]$, $[[0,x,y,0]]$, $[[x,0,0,y]]$, $[[x,0,y,0]]$:

\smallskip
$\displaystyle t(M;x,y)=
\sum_{\buildrel {\scriptstyle A\subseteq E}\over {\hbox{\sevenrm\ basis in\ }\scriptscriptstyle M}}
x^{\iota_M(A)}y^{\epsilon_M(A)}$.
\hfill (2)

\smallskip
(2) is the expansion in terms of Tutte activities of bases.

\bigskip\noindent
$\bullet$ \cite{GT90} Example 3.3\\ 
$[[x/2,x/2,y/2,y/2]]$:

\smallskip
$\displaystyle t(M;x,y)=
\sum_{A\subseteq E}
\bigl({x\over 2}\bigr)^{\raise 2pt \hbox{$\scriptstyle cr_M(A)+\iota_M(A)$}}
\bigl({y\over 2}\bigr)^{\raise 2pt \hbox{$\scriptstyle nl_M(A)+\epsilon_M(A)$}}$.
\hfill (3)

\medskip
The following 8 expansions (3.2)-(3.5) have been overlooked in \cite{GT90}.

\medskip
We say that $A\subseteq E$ is {\it internally} resp. {\it externally inactive} in $M$ if 
$\iota_M(A)=0$ resp. $\epsilon_M(A)=0$.

\eject
\medskip\noindent
$\bullet$ $[[x/2,x/2,y-1,1]]$, $[[x/2,x/2,1,y-1]]$:\\

\smallskip
$\displaystyle t(M;x,y)=
\sum_{A\subseteq E}
\bigl({x\over 2}\bigr)^{\raise 2pt \hbox{$\scriptstyle cr_M(A)+\iota_M(A)$}}
(y-1)^{nl_M(A)}$.
\hfill (3b)

\medskip\noindent
$\bullet$ $[[x/2,x/2,y,0]]$, $[[x/2,x/2,0,y]]$:\\

\smallskip
$\displaystyle t(M;x,y)=
\sum_{\buildrel {\scriptstyle A\subseteq E}\over {\hbox{\sevenrm\ externally inactive in\ }\scriptscriptstyle M}}
\bigl({x\over 2}\bigr)^{\raise 2pt \hbox{$\scriptstyle cr_M(A)+\iota_M(A)$}}
y^{nl_M(A)}$.
\hfill (3c)

\medskip\noindent
$\bullet$ $[[x-1,1,y/2,y/2]]$, $[[1,x-1,y/2,y/2]]$:\\

\smallskip
$\displaystyle t(M;x,y)=
\sum_{A\subseteq E}
(x-1)^{cr_M(A)}
\bigl({y\over 2}\bigr)^{\raise 2pt \hbox{$\scriptstyle nl_M(A)+\epsilon_M(A)$}}$.
\hfill (3d)

\medskip\noindent
$\bullet$ $[[x,0,y/2,y/2]]$, $[[0,x,y/2,y/2]]$:

\smallskip
$\displaystyle t(M;x,y)=
\sum_{\buildrel {\scriptstyle A\subseteq E}\over {\hbox{\sevenrm\ internally inactive in\ }\scriptscriptstyle M}}
x^{cr_M(A)}
\bigl({y\over 2}\bigr)^{\raise 2pt \hbox{$\scriptstyle nl_M(A)+\epsilon_M(A)$}}$.
\hfill (3e)

\medskip\noindent
$\bullet$ \cite{GT90} Example 3.4\\ 
$[[0,x,y-1,1]]$, $[[0,x,1,y-1]]$, $[[x,0,y-1,1]]$, $[[x,0,1,y-1]]$:

\smallskip
$\displaystyle t(M;x,y)=
\sum_{\buildrel {\scriptstyle A\subseteq E}\over {\hbox{\sevenrm\ spanning in\ }\scriptscriptstyle M}}
x^{\iota_M(A)}(y-1)^{nl_M(A)}$.
\hfill (4)

\medskip\noindent
$\bullet$ \cite{GT90} Example 3.5\\ 
$[[x-1,1,0,y]]$, $[[1,x,0,y]]$, $[[x-1,1,y,0]]$, $[[1,x-1,y,0]]$:

\smallskip
$\displaystyle t(M;x,y)=
\sum_{\buildrel {\scriptstyle A\subseteq E}\over {\hbox{\sevenrm\ independent in\ }\scriptscriptstyle M}}
(x-1)^{cr_M(A)}y^{\epsilon_M(A)}$.
\hfill (5)

\medskip
All the above expansions generalize to the 3-variable Tutte polynomials of matroid perspectives, see below Proposition \ref{prop:multiform}. 


\section{Derivatives for Matroid Perspectives}

Matroids abstract linear dependence in vector spaces.
Matroid morphisms, or {\it strong maps of matroids}, abstract linear mappings \cite{Ku86}, \cite{Ox11}.

\medskip
Matroids related by strong maps may have different ground sets.
However, using added loops, and up to a bijection if necessary, it can easily be seen that no generality is lost by considering matroids on a same ground set.
 
Two matroids $M$, $M'$ with a same set $E$ of elements constitute a {\it matroid perspective}, denoted here by $M\rightarrow M'$, if and only if we have at least one, hence all, of the following equivalent properties

\smallskip
$\bullet$ (MP1) every circuit of $M$ is a union of circuits of $M'$,

\smallskip
$\bullet$ (MP2) every cocircuit of $M'$ is a union of cocircuits of $M$,

\smallskip
$\bullet$ (MP2') every flat of $M'$ is a flat of $M$

\smallskip
$\bullet$ (MP3) no circuit of $M$ and cocircuit of $M'$ intersect in exactly one element.

\smallskip
$\bullet$ (MP4) for all $Y\subseteq X\subseteq E$, we have $r_{M'}(X)-r_{M'}(Y)\leq r_{M}(X)-r_{M}(Y)$

\smallskip
$\bullet$ (MP5) there is a matroid $N$ on a set $F$ with $E\subseteq F$ such that $M=N\setminus (F\setminus E)$ (deletion) and $M'=N/(F\setminus E)$ (contraction).

\medskip
Duality for matroid perspectives follows immediately from these equivalences.
We have 
$$M\rightarrow M' \hbox{ if and only if } M'^*\rightarrow M^*.$$

\bigskip
For $A\subseteq E$, we set $rcd_{M,M'}(A)=r(M)-r(M')-(r_M(A)-r_{M'}(A))$. 
We call $rcd_{M,M'}(A)$ the {\it rank codrop} of $A$ in the matroid perspective $M \rightarrow M'$.

\bigskip
We have defined in \cite{LV75}\cite{LV99} the Tutte polynomial of a matroid perspective $M\rightarrow M'$ on a set $E$ by 
$$t(M,M';x,y,z)=\sum_{A\subseteq E}(x-1)^{r(M')-r_{M'}(A)}(y-1)^{|A|-r_M(A)}z^{rcd_{M,M'}(A))}.$$

\medskip
We have plainly $t(M,M;x,y,z)=t(M;x,y)$.

\medskip
We point out that the Tutte polynomial of a matroid perspective may alternately be defined as the Tutte polynomial of any {\it major} $N$ of the pers\-pective, given by (MP5), pointed by the set $F\setminus E$ of extra elements (see \cite{LV99} Eq. (5.1)).
This point of view will not be used in the sequel.

\bigskip\medskip
Suppose $E$ is linearly ordered.
The following generalization of the basis expansion of matroid Tutte polynomials to matroid perspectives holds \cite{LV80}\cite{LV99}:
$$t(M,M';x,y,z)=
\sum_
{\buildrel {A\subseteq E}\over
{\buildrel {\hbox{\sevenrm\ spanning in\ }M'}
            \over{\scriptscriptstyle\hbox{\sevenrm\ independent in\ } M}}} 
x^{\iota_{M'}(A)}y^{\epsilon_M(A)}z^{rcd_{M,M'}(A))}.$$

\bigskip
The main result of the paper extends this expansion to partial derivatives.

\medskip
\begin{theorem}
\label{thm:perspective_derivative}
\noindent
Let $M\rightarrow M'$ be a matroid perspective on a linearly ordered set $E$,
and $p$, $q$ be non negative integers.
Then
$${{\partial^{p+q} t}\over{{\partial x^p}{\partial y^q}}}                                         
(M,M';x,y,z)
=p!q!\sum_{\buildrel {A\subseteq E}\over
{\buildrel\scriptstyle {cr_{M'}(A)=p}\over{\scriptstyle nl_M(A)=q}}}
x^{\iota_{M'}(A)}y^{\epsilon_M(A)}z^{rcd_{M,M'}(A))}.$$
\end{theorem}

\bigskip
The proof of Theorem \ref{thm:perspective_derivative} uses the Dawson partition associated with a matroid perspective.

\bigskip
Let $M\rightarrow M'$ be a matroid perspective on linearly ordered set $E$, and 
$${\cal P}:=\{X\subseteq E \mid
\hbox{\it $X$ is independent in $M$ and spanning in $M'$}\}$$
\medskip
For $A\subseteq E$, set $$f_{M,M'}(A):=A\cup P_{M'}(A)\setminus Q_M(A)$$
\indent
Then $$B:=f_{M,M'}(A)\in{\cal P}$$ defines the interval $$[B\setminus\Int_{M'}(B),B\cup\Ext_M(B)]$$ containing $A$ of the Dawson partition defined by $\cal P$ (see below Proposition \ref{prop:Dawson_partition}).

\medskip
\begin{lemma}
\label{lem:independent/spanning}
(i) The set $B$ is independent in $M$ and spanning in $M'$.

(ii) The Boolean interval $[B\setminus\Int_{M'}(A),B\cup\Ext_M(A)]$ contains $A$.
\end{lemma}

\begin{proof}
(i) We show that $B$ is independent in $M$.
If not there is a circuit $C$ of $M$ contained in $B=A\cup P_{M'}(A)\setminus Q_M(A)$.
We have $C\cap P_{M'}(A)\not=\emptyset$.
Otherwise, we have $C\subseteq A\setminus Q_M(A)\subseteq A$, 
Therefore, by definition of $Q_M(A)$,
the smallest element of $C$ would be in $Q_M(A)$, 
contradicting $C\cap Q_M(A)=\emptyset$.
Let $e$ be the greatest element of $C\cap P_{M'}(A)$.
By definition of $P_{M'}(A)$, there is a cocircuit $D$ of $M'$ with smallest element $e$ such that $D\subseteq E\setminus A$.
By the choice of $e$, and since $e$ is smallest in $D$, we have $C\cap D=\{e\}$, contradicting (MP3).

The proof that $B$ is spanning in $M'$ is obtained in the same way by using matroid perspective duality.

(ii) We prove that $A\subseteq B\cup\Ext_M(B)$,
by showing that $Q_M(A)\subseteq\Ext_M(B)$.
Let $e\in Q_M(A)$.
By definition, $e$ is smallest in a circuit $X\subseteq A$.
Suppose $X$ chosen such that the second smallest element of $f\in X\cap Q_M(A)$, if such an element exists, is the greatest possible.
If $f$ does not exist, we set $C=X$.
Otherwise, by definition of $Q_M(A)$, there is a circuit $Y$ of $M$ with smallest element $f$ such that $Y\subseteq A$.
Since $e\not\in Y$, by elimination of $f$ from $X$ and $Y$, there is a circuit $C$ of $M$ containing $e$ such that $C\subseteq (X\cup Y)\setminus\{f\}$.
By construction, the element $e$ is smallest in $C$ and the second smallest element $f'\in C\cap Q_M(A)$, if one exists, is $>f$.
Hence $f'$ does not exist, by the choice of $X$, 
and we have $C\cap Q_M(A)=\{e\}$.
It follows that $e\in\Ext_M(A\setminus Q_M(A)$, hence $e\in\Ext_M(B)$.

The proof of the second inclusion is obtained in the same way by using matroid perspective duality.
\end{proof}

\begin{lemma}
\label{lem:interval}
Conversely, let $B\subseteq E$ be independent in $M$ and spanning in $M'$, 
and $A$ be in the Boolean interval $[B\setminus\Int_{M'}(B),B\cup\Ext_M(B)]$.

Then, we have 

\smallskip\hspace{0.5cm}(i) $\Int_{M'}(A)=\Int_{M'}(B)\cap A$,

\smallskip\hspace{0.5cm}(ii) $P_{M'}(A)=\Int_{M'}(B)\setminus A$,

\smallskip\hspace{0.5cm}(iii) $\Ext_M(A)=\Ext_M(B)\setminus A$,

\smallskip\hspace{0.5cm}(iv) $Q_M(A)=\Ext_M(B)\cap A$.
\end{lemma}

\bigskip
\begin{figure}[h] 
\centerline{\includegraphics[scale=0.65]{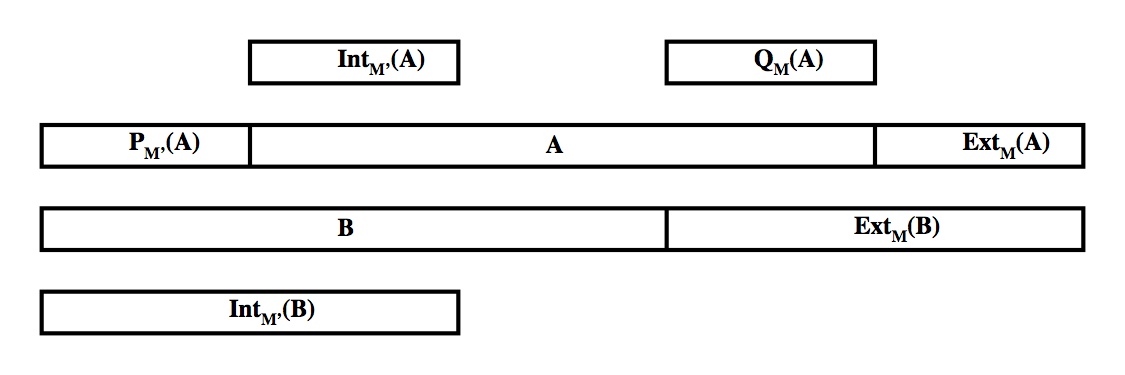}}
\caption{scheme of the Dawson interval containing $A$}
\end{figure}

\vspace{0cm}
\begin{proof}
(i) Let $e\in\Int_{M'}(A)$.
By definition, there is a cocircuit $D$ of $M'$ with smallest element $e\in A$ such that $e\in D\subseteq (E\setminus A)\cup\{e\}$.

We show that $e\in B$.
If not, we have $e\in\Ext_M(B)$. 
Let $C$ be the fundamental circuit of $e$ in $M$ with respect to $B$.
Then $e$ is the smallest element of $C$, and we have $e\in C\subseteq B\cup\{e\}$.
Necessarily, we have $C\cap\Int_{M'}(B)=\emptyset$.
Otherwise, let $f\in C\cap\Int_{M'}(B)$.
Then $f$ is the smallest element of its fundamental cocircuit $X$ in $M'$ with respect to $B$. 
We have $X\subseteq (E\setminus B)\cup\{f\}$.
Hence $f\in C\cap X\subseteq\{e,f\}$.
By (MP3), it follows that $C\cap X=\{e,f\}$,
which is impossible since $e$ is the smallest element of $C$, and $f$ is the smallest element of $X$.

We show that $e\in\Int_{M'}(B)$.
If not, we have $(D\setminus\{e\})\cap B\not=\emptyset$.
Let $f\in(D\setminus\{e\})\cap B$.
Since $D\setminus\{e\}\subseteq E\setminus A$, 
we have $f\in(\Int_{M'}(B)\cup\Ext_M(B))\setminus A$.
If $f\in\Ext_M(B)$, let $C$ be the fundamental circuit of $f$ in $M$ with respect to $B$.
As above since 
$C\subseteq (B\cup\{f\})\setminus\Int_{M'}(B)$ and 
$D\subseteq (E\setminus A)\cup\{e\}$ with 
$(B\setminus\Int_{M'}(B)\subseteq A\subseteq (B\cup\Ext_M(B))$, we have
$f\in (C\cap D)\subseteq\{e,f\}$.
Hence by (MP3), necessarily $C\cap D=\{e,f\}$.
But we get a contradiction, since $e$ is the smallest element of $D$, and $f$ is the smallest element of $C$.
It follows that $e\in\Int_{M'}(B)$.

We have proved that $\Int_{M'}(A)\subseteq\Int_{M'}(B)\cap A$.
For the reverse inclusion, consider $e\in\Int_{M'}(B)\cap A$.
The element $e$ is smallest in its fundamental cocircuit $D$ in $M'$ with respect to $B$, such that $e\in D\subseteq (E\setminus B)\cup\{e\}$.
We have $D\cap\Ext_M(B)=\emptyset$.
Otherwise there is $x\in D\cap\Ext_M(B)$, and the fundamental circuit $C$ of $x$ in $M$ with respect to $B$ has $x$ as smallest element and $x\in C\subseteq B\cup\{x\}$.
We have $x\in C\cap D\subseteq\{e,x\}$, hence $C\cap D=\{e,x\}$ by (MP3),
yielding a contradiction since $e$ is smallest in $D$ and $x$ is smallest in $C$.
Since $A\subseteq B\cup\Ext_M(B)$, 
it follows that $D\subseteq (E\setminus A)\cup\{e\}$.
Hence $e\in\Int_{M'}(A)$.

(ii) We have $P_{M'}(A)\subseteq\Int_{M'}(B)$ by construction.
We show that $e\in\Int_{M'}(B)\setminus A\subseteq P_{M'}(A)$.
Let $e\in\Int_{M'}(B)\setminus A$.
There is a cocircuit $D$ of $M'$ with smallest element $e$ such that 
$e\in D\subset E\setminus B$.
If $D\cap A=\emptyset$, we have $e\in P_{M'}(A)$.
Otherwise, let $x\in D\cap A$.
Necessarily, we have $x\in A\setminus B\subseteq\Ext_M(B)$.
Hence, there is a circuit $C$ of $M$ with smallest element $x$ such that $x\in C\subseteq B\cup\{x\}$.
We have $x\in C\cap D\subseteq\{e,x\}$, hence $C\cap D=\{e,x\}$ by (MP3),
yielding a contradiction since $e$ is smallest in $D$ and $x$ is smallest in $C$.

(iii)(iv) The proofs are obtained by matroid perspective duality from the proofs of (i)(ii).
\end{proof}

\bigskip
Reassembling Lemmas \ref{lem:independent/spanning} and \ref{lem:interval}, we get the main properties of the Dawson partition for matroid perspectives.

\medskip
\begin{proposition}
\label{prop:Dawson_partition}
Let $M\rightarrow M'$ be a matroid perspective on a linearly ordered set $E$.
Let $A$ be a subset of $E$, and $B:=f_{M,M'}(A)$.

Then $B$ is spanning in $M'$ and independent in $M$,
and $B\setminus\Int_{M'}(B)\subseteq A\subseteq B\cup\Ext_M(B)$. 

We have $\Int_{M'}(B)=\Int_{M'}(A)\cup P_{M'}(A)$  and
$\Ext_M(B)=\Ext_M(A)\cup Q_M(A)$.

\medskip
Conversely, let $B\subseteq E$ be spanning in $M'$ and independent in $M$.
Let $P\subseteq\Int_{M'}(B)$  and $Q\subseteq\Ext_M(B)$.
Set $A=B\setminus P\cup Q$.
Then, we have $B=f_{M,M'}(A)$, $P_{M'}(A)=P$, $\Int_{M'}(A)=Int_{M'}(B)\setminus P$, $Q_M(A)=Q$, and $\Ext_M(A)=\Ext_M(B)\setminus Q$.

\medskip
An independent/spanning set $B$ is the unique independent/spanning set in the interval $[B\setminus\Int_{M'}(B),B\cup\Ext_M(B)]$,
Moreover, the intervals associated with two independent/spanning sets are disjoint,
and the collection of these intervals for all independent/spanning sets of $M\rightarrow M'$ constitute a partition of $2^E$.

\medskip
The Dawson interval $[g_{M,M'}(A),h_{M,M'}(A)]$ containing $A$ is defined by
$$g_{M,M'}(A)=A\setminus\Int_{M'}(A)\setminus Q_M(A),$$
$$h_{M,M'}(A)=A\cup\Ext_M(A)\cup P_{M'}(A).$$
\hfill\square
\end{proposition}

\bigskip\noindent
{\bf Example 2}

\begin{figure}[h] 
\centerline{\includegraphics[scale=0.35]{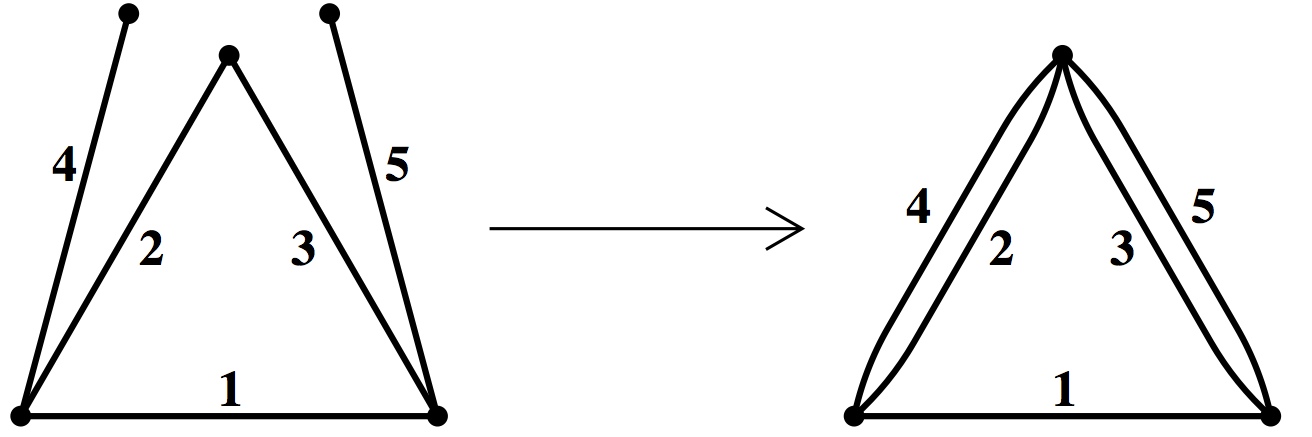}}
\end{figure}

The Tutte polynomial of the graphic matroid perspective defined by the cycle spaces of the above graphs by identifying vertices is
$t(x,y,z)=(x^2+3x+y+3)z^2+(2x+2y+5)z+y+2$. We have

\medskip
\centerline{
\scalebox{0.55}{
\(
\null\vspace{3cm}
\begin{array}{|c|c|c|c|c|c|c|c|c|c|c|c|c|c|c|c|c|c|c|c|c|c|}
\hline
\hbox{ind./span.}&
\begin{array}{c} \phantom{} \\ \phantom{} \end{array} 12\begin{array}{c} \phantom{} \\ \phantom{} \end{array} &13&14&15&23&25&34&45&124&135&234&235&125&134&145&245&345&2345&1245&1345
\\ \hline
\hbox{activities}& 12\ * & 1\ * & 1\ * & 1\ * & *\ 1&&&&1\ *&1\ *&*\ 1&*\ 1&&&&&&*\ 1&&\begin{array}{c} \phantom{} \\ \phantom{} \end{array}
\\ \hline
t&
\begin{array}{c} 12 \\ x^2z^2 \end{array}&
\begin{array}{c} 13 \\ xz^2 \end{array}&
\begin{array}{c} 14 \\ xz^2 \end{array}&
\begin{array}{c} 15 \\ xz^2 \end{array}&
\begin{array}{c} 23 \\ yz^2 \end{array}&
\begin{array}{c} 25 \\ z^2 \end{array}&
\begin{array}{c} 34 \\ z^2 \end{array}&
\begin{array}{c} 45 \\ z^2 \end{array}&
\begin{array}{c} 124 \\ xz \end{array}&
\begin{array}{c} 135 \\ xz \end{array}&
\begin{array}{c} 234 \\ yz \end{array}&
\begin{array}{c} 235 \\ yz \end{array}&
\begin{array}{c} 125 \\ z \end{array}&
\begin{array}{c} 134 \\ z \end{array}&
\begin{array}{c} 145 \\ z \end{array}&
\begin{array}{c} 245 \\ z \end{array}&
\begin{array}{c} 345 \\ z \end{array}&
\begin{array}{c} 2345 \\ z \end{array}&
\begin{array}{c} 1245 \\ 1 \end{array}&
\begin{array}{c} 1345 \\ 1 \end{array}\\ \hline
{\partial t}\over{\partial x}&
\begin{array}{cc} 
  \begin{array}{c} 1 \\ xz^2 \end{array}&
  \begin{array}{c} 2 \\ xz^2 \end{array} 
\end{array}&
\begin{array}{c} 3 \\ z^2 \end{array}&
\begin{array}{c} 4 \\ z^2 \end{array}&
\begin{array}{c} 5 \\ z^2 \end{array}&
\begin{array}{c} \\ \end{array}&
\begin{array}{c} \\ \end{array}&
\begin{array}{c} \\ \end{array}&
\begin{array}{c} \\ \end{array}&
\begin{array}{c} 24 \\ z \end{array}&
\begin{array}{c} 35 \\ z \end{array}&
\begin{array}{c} \\ \end{array}&
\begin{array}{c} \\ \end{array}&
\begin{array}{c} \\ \end{array}&
\begin{array}{c} \\ \end{array}&
\begin{array}{c} \\ \end{array}&
\begin{array}{c} \\ \end{array}&
\begin{array}{c} \\ \end{array}&
\begin{array}{c} \\ \end{array}&
\begin{array}{c} \\ \end{array}&\\ \hline
{\partial t}\over{\partial y}&
\begin{array}{c} \\ \end{array}&
\begin{array}{c} \\ \end{array}&
\begin{array}{c} \\ \end{array}&
\begin{array}{c} \\ \end{array}&
\begin{array}{c} 123 \\ z^2 \end{array}&
\begin{array}{c} \\ \end{array}&
\begin{array}{c} \\ \end{array}&
\begin{array}{c} \\ \end{array}&
\begin{array}{c} \\ \end{array}&
\begin{array}{c} \\ \end{array}&
\begin{array}{c} 1234 \\ z \end{array}&
\begin{array}{c} 1235 \\ z \end{array}&
\begin{array}{c} \\ \end{array}&
\begin{array}{c} \\ \end{array}&
\begin{array}{c} \\ \end{array}&
\begin{array}{c} \\ \end{array}&
\begin{array}{c} \\ \end{array}&
\begin{array}{c} 12345 \\ 1 \end{array}&
\begin{array}{c} \\ \end{array}&
\begin{array}{c} \\ \end{array}\\ \hline
{1\over 2}{{\partial^2 t}\over{\partial x^2}}&
\begin{array}{c} \emptyset \\ z^2 \end{array}&
\begin{array}{c} \\ \end{array}&
\begin{array}{c} \\ \end{array}&
\begin{array}{c} \\ \end{array}&
\begin{array}{c} \\ \end{array}&
\begin{array}{c} \\ \end{array}&
\begin{array}{c} \\ \end{array}&
\begin{array}{c} \\ \end{array}&
\begin{array}{c} \\ \end{array}&
\begin{array}{c} \\ \end{array}&
\begin{array}{c} \\ \end{array}&
\begin{array}{c} \\ \end{array}&
\begin{array}{c} \\ \end{array}&
\begin{array}{c} \\ \end{array}&
\begin{array}{c} \\ \end{array}&
\begin{array}{c} \\ \end{array}&
\begin{array}{c} \\ \end{array}&
\begin{array}{c} \\ \end{array}&
\begin{array}{c} \\ \end{array}&\\ \hline
\end{array}
\)
}}

\medskip
\scalebox{0.8}{
\indent
Table 4 shows 1) the Dawson partition associated with the matroid perspective} 

\vspace{-3pt}
\scalebox{0.8}{
- there are 20 independent/spanning sets, defining 20 pairwise disjoint intervals of}

\vspace{-3pt}
\scalebox{0.8}{
respectively $4+2+2+2+2+1+1+1+1+1+2+2+2+2+1+1+2+1+1+1$} 

\vspace{-3pt}
\scalebox{0.8}{
elements, and 2) the relationship between derivative terms and subsets of elements.}

\smallskip
\centerline{Table 4}

\bigskip
\begin{remark}
\label{rem:Dawson_interval}
{\rm By (MP1) and (MP2), we have $\Int_{M'}(A)\subseteq \Int_M(A)$,
$P_{M'}(A)\subseteq P_M(A)$, $\Ext_M(A)\subseteq \Ext_{M'}(A)$,
$Q_M(A)\subseteq Q_{M'}(A)$.

It follows from Proposition \ref{prop:Dawson_partition} that
$g_M(A),g_{M'}(A)\subseteq g_{M,M'}(A)$ and \linebreak
$h_{M,M'}(A)\subseteq h_M(A),h_{M'}(A)$.
Therefore, we have 
$$[g_{M,M'}(A),h_{M,M'}(A)]\subseteq 
[g_M(A),h_M(A)]\cap [g_{M'}(A),h_{M'}(A)].$$
\indent
Hence, {\it the Dawson intervals in $M$ and in $M'$ are unions of Dawson intervals in $M\rightarrow M'$}.

\medskip
In Example 2, consider the Dawson interval of $M\rightarrow M'$ defined by the independent/spanning set $124$, namely $[24,124]$.
The Dawson intervals containing $124$ are respectively $[\emptyset,1245]$ in $M$ and $[4,124]$ in $M'$.
The intersection of these two intervals is $[4,124]$.
It contains strictly $[24,124]$.
Therefore, the Dawson partition of $M\rightarrow M'$ is not the meet of the Dawson partitions of $M$ and $M'$, and, most probably, cannot be simply constructed from them.}

\end{remark}

\bigskip
For non negative integers $i,j,k,p,q$,
let $a_{ijkpq}$ be the number of subsets $A$ of $E$ such that
$\iota_{M'}(A)=i$, $\epsilon_M(A)=j$, $r(M)-r(M')-(r_M(A)-r_{M'}(A))=k$,
$r(M')-r_{M'}(A)=p$ and $|A|-r_M(A)=q$,
and $b_{ijk}$ be the number of subset $B$ of $E$ spanning in $M'$
and independent in $M$ such that $\iota_{M'}(B)=i$,
$\epsilon_M(B)=j$ and $r(M)-r(M')-(r_M(B)-r_{M'}(B))=k$.

\medskip
It follows from Proposition \ref{prop:Dawson_partition} that

\medskip
\begin{lemma}
\label{lem:ab_formula}
For $p\leq i$ and $q\leq j$

\smallskip\hspace{3.7cm}
$\displaystyle a_{i-p,j-q,k,p,q}={i\choose p}{j \choose q}b_{ijk}$.
\hfill\square
\end{lemma}

\medskip
\begin{proof}[Proof of Theorem \ref{thm:perspective_derivative}]
By differentiating (4), and using Lemma \ref{lem:ab_formula}, we get\\

\noindent
$\displaystyle{
{{\partial^{p+q} t}\over{{\partial x^p}{\partial y^q}}}(M,M';x,y,z)=}$\\
\null\hspace{1cm} 
$\displaystyle{=\sum_{i\geq p,j\geq q}
i(i-1)\ldots(i-p+1)
j(j-1)\ldots(j-q+1)
b_{ijk}x^{i-p}y^{j-q}z^k=}$\\
\null\hspace{1cm} 
$\displaystyle{=\sum_{i\geq p,j\geq q}p!q!a_{i-p,j-q,k,p,q}x^{i-p}
y^{j-q}z^k
}$.
\end{proof}

\bigskip
As in Corollary \ref{cor:m_alternate}, dualities provide alternate expansions of partial derivatives.
They will be stated below as Corollary \ref{cor:p_alternate} of Theorem \ref{thm:dualities}.

\medskip
\begin{corollary}
\label{cor:perspective_(x,x)_derivative}
\noindent
Let $M\rightarrow M'$ be a matroid perspective on a linearly ordered set $E$, and $p$ be a non negative integer.
Then
$${{d^p t}\over{dx^p}}(M,M';x,x,z)=
p!\sum_{\buildrel {A\subseteq E}\over
{\scriptstyle cr_{M'}(A)+nl_M(A)=p}}x^{\iota_{M'}(A)+\epsilon_M(A)}z^{rcd_{M,M'}(A)}.$$
\end{corollary}

\begin{proof}
By Taylor's formula, and Theorem \ref{thm:perspective_derivative} we have
$$t(M,M';x+u,y+v,z)
=\sum_{p\geq 0 q\geq 0}u^pv^q\sum_{\buildrel {A\subseteq E}\over
{\buildrel\scriptstyle {cr_{M'}(A)=p}\over{\scriptstyle nl_M(A)=q}}}
x^{\iota_{M'}(A)}y^{\epsilon_M(A)}z^{rcd_{M,M'}(A))}.$$
Hence
$$t(M,M';x+u,x+u,z)
=\sum_{p\geq 0 q\geq 0}u^{p+q}\sum_{\buildrel {A\subseteq E}\over
{\buildrel\scriptstyle {cr_{M'}(A)=p}\over{\scriptstyle nl_M(A)=q}}}
x^{\iota_{M'}(A)}y^{\epsilon_M(A)}z^{rcd_{M,M'}(A))}.$$
By Taylor's formula again
$${{d^p t}\over{dx^p}}(M,M';x+u,x+u,z)=\sum_{k\geq 0}{1\over k!}u^k{{d^p t}\over{dx^p}}(M,M';x,x,z),$$
and Corollary \ref{cor:perspective_(x,x)_derivative} follows.
\end{proof}

\bigskip
The following useful 5-variable expansion of the Tutte polynomial can be readily obtained from Theorem \ref{thm:perspective_derivative}.

\medskip
\begin{theorem}
\label{thm:5_variables}
Let $M\rightarrow M'$ be a matroid perspective on a linearly ordered set $E$. 
The following identity holds:
$$t(M,M';x+u,y+v,z)=\sum_{A\subseteq E}x^{\iota_{M'}(A)}u^{cr_{M'}(A)}y^{\epsilon_M(A)}
v^{nl_M(A)}z^{rcd_{M,M'}(A)}.$$
\end{theorem}

\bigskip
\begin{proof}
By Taylor's formula, we have
$$\displaystyle
t(M,M';x+u,y+v,z)=\sum_{p\geq 0\ q\geq 0}{1\over{p!q!}}u^pv^q
{{\partial^{p+q} t}\over{{\partial x^p}{\partial y^q}}}(M,M';x,y,z).$$
\indent
We obtain Theorem \ref{thm:5_variables} by substituting in this identity the expression for
${{\partial^{p+q} t}/{{\partial x^p}{\partial y^{q}}}}$ given by Theorem \ref{thm:perspective_derivative}.
\end{proof}

\bigskip
An example for Theorem \ref{thm:5_variables} is given in Table 5.

\medskip
\noindent
{\bf Example 2} (continued)

\bigskip
\centerline{
\scalebox{1}{
\(
\begin{array}{|l||l|l|l|l|l||l|||l||l|l|l|l|l||l||}
 \hline
\scalebox{0.75}{}& 
\scalebox{0.75}{$\iota'$} & 
\scalebox{0.75}{$cr'$} & 
\scalebox{0.75}{$\epsilon$} & 
\scalebox{0.75}{$nl$} & 
\scalebox{0.75}{$rcd$} & &
& 
\scalebox{0.75}{$\iota'$} & 
\scalebox{0.75}{$cr'$} & 
\scalebox{0.75}{$\epsilon$} & 
\scalebox{0.75}{$nl$} & 
\scalebox{0.75}{$rcd$} & 
\\ \hline\hline
\emptyset &2&0&0&0&2&\r{$x^2z^2$}&123&0&0&0&1&2&\g{$vz^2$}\\
1&1&1&0&0&2&\r{$xuz^2$}&124&1&0&0&0&1&\b{$xz$}\\
2&1&1&0&0&2&\r{$xuz^2$}&125&0&0&0&0&1&z\\
3&0&1&0&0&2&\y{$uz^2$}&134&0&0&0&0&1&z\\
4&0&1&0&0&2&\y{$uz^2$}&135&1&0&0&0&1&\b{$xz$}\\
5&0&1&0&0&2&\y{$uz^2$}&145&0&0&0&0&1&z\\
12&2&0&0&0&2&\r{$x^2z^2$}&234&0&0&1&0&1&\m{$yz$}\\
13&1&0&0&0&2&\y{$xz^2$}&235&0&0&1&0&1&\m{$yz$}\\
14&1&0&0&0&2&\y{$xz^2$}&245&0&0&0&0&1&z\\
15&1&0&0&0&2&\y{$xz^2$}&345&0&0&0&0&1&z\\
23&0&0&1&0&2&\g{$yz^2$}&1234&0&0&0&1&1&\m{$vz$}\\
24&0&1&0&0&1&\b{$uz$}&1235&0&0&0&1&1&\m{$vz$}\\
25&0&0&0&0&2&z^2&1245&0&0&0&0&0&1\\
34&0&0&0&0&2&z^2&1345&0&0&0&0&0&1\\
35&0&1&0&0&1&\b{$uz$}&2345&0&0&1&0&1&\t{$y$}\\
45&0&0&0&0&2&z^2&12345&0&0&0&1&0&\t{$v$}\\ \hline
\end{array}
\)
}}

\medskip
\scalebox{0.8}{
\indent
Table 5 shows the expansion of 
$t(x+u,y+v,z)=$}

\scalebox{0.8}{
$=\bigl(\r{$(x+u)^2$}+\y{$3(x+u)$}+\g{$y+v$}+3\bigr)z^2+
\bigl(\b{$2(x+u)$}+\m{$2(y+v)$}+5\bigr)z+\t{$y+v$}+2$.}

\smallskip
\centerline{Table 5}

\medskip 
As in Section 1, by specializing $x$, $y$, $u$, $v$ to 0 or 1, or replacing by $x/2$ or $y/2$, after an appropriate change of variables, we get from Proposition \ref{thm:5_variables} 25=4+4+1+2+2+2+2+4+4 remarkable expansions of $t(M,M';x,y,z)$, classified in 9 families with different summands.
In the next proposition, we list these 9 families by exhibiting one representative for each.

\begin{proposition}
\label{prop:multiform}
\null\ 

\smallskip\noindent
$\hbox{\rm (1)}\ [[x-1,1,y-1,1]]$, $[[x-1,1,1,y-1]]$, $[[1,x-1,y-1,1]]$, $[[1,x-1,1,y-1]]${\rm :}
$$t(M,M';x,y,z)=
\sum_{A\subseteq E}
(x-1)^{cr_{M'}(A)}(y-1)^{nl_M(A)}z^{rcd_{M,M'}(A)}.$$

\smallskip\noindent
$\hbox{\rm (2)}\ [[0,x,0,y]]$, $[[0,x,y,0]]$, $[[x,0,0,y]]$, $[[x,0,y,0]]${\rm :}
$$t(M,M';x,y,z)=
\sum_
{\buildrel {A\subseteq E}\over
{\buildrel {\hbox{\sevenrm\ spanning in\ }M'}
            \over{\scriptscriptstyle\hbox{\sevenrm\ independent in\ } M}}} 
x^{\iota_{M'}(A)}y^{\epsilon_M(A)}z^{rcd_{M,M'}(A)}.$$

\smallskip\noindent
$\hbox{\rm (3)}\ [[x/2,x/2,y/2,y/2]]${\rm :}
$$t(M,M';x,y,z)=
\sum_{A\subseteq E}
\bigl({x\over 2}\bigr)^{\raise 2pt \hbox{$\scriptstyle cr_{M'}(A)+\iota_{M'}(A)$}}
\bigl({y\over 2}\bigr)^{\raise 2pt \hbox{$\scriptstyle nl_M(A)+\epsilon_M(A)$}}
z^{rcd_{M,M'}(A)}.$$

\smallskip\noindent
$\hbox{\rm (3b)}\ [[x/2,x/2,y-1,1]]$, $[[x/2,x/2,1,y-1]]${\rm :}\\
$$t(M,M';x,y,z)=
\sum_{A\subseteq E}
\bigl({x\over 2}\bigr)^{\raise 2pt \hbox{$\scriptstyle cr_{M'}(A)+\iota_{M'}(A)$}}
(y-1)^{nl_M(A)}z^{rcd_{M,M'}(A)}.$$

\smallskip\noindent
$\hbox{\rm (3c)}\ [[x/2,x/2,y,0]]$, $[[x/2,x/2,0,y]]${\rm :}\\
$$ t(M;x,y)=
\sum_{\buildrel {\scriptstyle A\subseteq E}\over {\hbox{\sevenrm\ externally inactive in\ }\scriptscriptstyle M}}
\bigl({x\over 2}\bigr)^{\raise 2pt \hbox{$\scriptstyle cr_{M'}(A)+\iota_{M'}(A)$}}
y^{nl_M(A)}z^{rcd_{M,M'}(A)}.$$

\smallskip\noindent
$\hbox{\rm (3d)}\ [[x-1,1,y/2,y/2]]$, $[[1,x-1,y/2,y/2]]${\rm :}\\
$$ t(M,M';x,y,z)=
\sum_{A\subseteq E}
(x-1)^{cr_{M'}(A)}
\bigl({y\over 2}\bigr)^{\raise 2pt \hbox{$\scriptstyle nl_M(A)+\epsilon_M(A)$}}
z^{rcd_{M,M'}(A)}.$$

\smallskip\noindent
$\hbox{\rm (3e)}\ [[x,0,y/2,y/2]]$, $[[0,x,y/2,y/2]]${\rm :}
$$ t(M,M';x,y,z)=
\sum_{\buildrel {\scriptstyle A\subseteq E}\over {\hbox{\sevenrm\ internally inactive in\ }\scriptscriptstyle M'}}
x^{cr_{M'}(A)}
\bigl({x\over 2}\bigr)^{\raise 2pt \hbox{$\scriptstyle nl_M(A)+\epsilon_M(A)$}}
z^{rcd_{M,M'}(A)}.$$

\smallskip\noindent
$\hbox{\rm (4)}\ [[0,x,y-1,1]]$, $[[0,x,1,y-1]]$, $[[x,0,y-1,1]]$, $[[x,0,1,y-1]]${\rm :}
$$ t(M,M';x,y,z)=
\sum_{\buildrel {\scriptstyle A\subseteq E}\over {\hbox{\sevenrm\ spanning in\ }\scriptscriptstyle M'}}
x^{\iota_{M'}(A)}(y-1)^{nl_M(A)}z^{rcd_{M,M'}(A)}.$$

\smallskip\noindent
$\hbox{\rm (5)}\ [[x-1,1,0,y]]$ $[[1,x,0,y]]$, $[[x-1,1,y,0]]$, $[[1,x-1,y,0]]${\rm :}

\smallskip\hspace{0.8cm}
$\displaystyle t(M,M';x,y,z)=
\sum_{\buildrel {\scriptstyle A\subseteq E}\over {\hbox{\sevenrm\ independent in\ }\scriptscriptstyle M}}
(x-1)^{cr_{M'}(A)}y^{\epsilon_M(A)}z^{rcd_{M,M'}(A)}$.
\hfill\square
\end{proposition}

\bigskip

\section{Dualities}

\medskip
As observed by G. Gordon and L. Traldi for matroids, in consequence of \cite{GT90} Theorem 2, certain expansions of Tutte polynomials of matroids contain the same summands (see the last paragraph of \cite{GT90} Example 3.1).
This pro\-perty follows from the existence of dualities - see also \cite{Da81}. 
These dualities generalize to matroid perspectives.
In this section, we exhibit precisely those matroid perspective dualities, between corank and internal activity on one hand, and nullity and external activity on the other.

\begin{lemma}
\label{lem:special_circuit}
Let $M$ be a matroid on a linearly ordered set $E$, and $A\subseteq E$.

Let $e\in\Ext_M(A)\cup Q_M(A)$.
There is a circuit $C$ of $M$ with smallest element $e$ such that $e\in C\subseteq (A\setminus Q_M(A))\cup\{e\}$.
\end{lemma}

\begin{proof}
There is a circuit $C$ of $M$ with smallest element $e$ such that $e\in C\subseteq A\cup\{e\}$.
Suppose that $C$ is chosen with this property such that the smallest element $x\in C\cap Q_M(A)$ if any is the greatest possible.
There is a circuit $X$ of $M$ with smallest element $x$ such that $x\in X\subseteq A$.
By elimination, there is a circuit $C'$ such that $e\subseteq (C\cup X)\setminus\{x\}$.
Then $e$ is the smallest element of $C'$, and we have $e\in C'\subseteq A\cup\{e\}$.
Furthermore, the smallest element in $C'\cap Q_M(A)$ if any is $>x$, 
contradicting the choice of $C$.
Therefore, $C\cap Q_M(A)=\emptyset$.
\end{proof}

\begin{proposition}
\label{prop:circuit_duality}
Let $M\rightarrow M'$ be a matroid perspective on a linearly ordered set $E$, and $A\subseteq E$.
Set $A'=A\setminus Q_M(A)\cup\Ext_M(A)$.
Then we have 
$$P_{M'}(A')=P_{M'}(A),$$ 
$$\Int_{M'}(A')=\Int_{M'}(A),$$ $$Q_M(A')=\Ext_M(A),$$ 
and $$\Ext_M(A')=Q_M(A).$$ 
\end{proposition}

\begin{proof}
A circuit $C$ of $M$ with smallest element $e\in\Ext_M(A)\cup Q_M(A)$ 
such that $C\setminus\{e\}\subseteq A$ 
and a cocircuit $D$ of $M'$ with smallest element $f\in\Int_{M'}(A)\cup P_{M'}(A)$ 
such that $D\setminus\{f\}\subseteq E\setminus A$ have an empty intersection.
Otherwise, by (MP3), we have $C\cap D=\{e,f\}$, but then $e<f$ since $e$ is smallest in $C$, and $f<e$ since $f$ is smallest in $D$.

It follows that $\Int_{M'}(A')=\Int_{M'}(A)$ and $P_{M'}(A')=P_{M'}(A)$.
Using Lemma \ref{lem:special_circuit}, we obtain that $\Ext_M(A')=Q_M(A)$ $Q_M(A')=\Ext_M(A)$.
\end{proof}

Using matroid perspective duality, Proposition \ref{prop:cocircuit_duality} follows from Proposition \ref{prop:circuit_duality}.

\begin{proposition}
\label{prop:cocircuit_duality}
Let $M\rightarrow M'$ be a matroid perspective on a linearly ordered set $E$, and $A\subseteq E$.
Set $A'=A\cup P_{M'}(A)\setminus\Int_{M'}(A)$.
Then we have 
$$P_{M'}(A')=\Int_{M'}(A),$$ 
\vspace{-7mm} $$\Int_{M'}(A')=P_{M'}(A),$$  
$$Q_M(A')=\Ext_M(A),$$ 
and\\
\vspace{8pt}\hspace{4.9cm}
$\displaystyle \Ext_M(A')=Q_M(A).$ 
\hfill\square
\end{proposition}

Theorem \ref{thm:dualities} then follows Proposition \ref{prop:circuit_duality} and Proposition \ref{prop:cocircuit_duality}.

\begin{theorem}
\label{thm:dualities}
The mappings $\phi^*_{M,M'}:2^E\mapsto 2^E$ and $\phi_{M,M'}:2^E\mapsto 2^E$, defined by 
$$\phi^*_{M,M'}(A)=A\cup P_{M'}(A)\setminus\Int_{M'}(A)$$ and 
$$\phi_{M,M'}(A)=A\setminus Q_M(A)\cup\Ext_M(A),$$
are two involutions of $2^E$, 
exchanging $cr_{M'}(A)$ and $\Int_{M'}(A)$ 
resp. $nl_M(A)$ and $\Ext_M(A)$.
\hfill\square
\end{theorem}

\medskip
Applying Theorem \ref{thm:dualities} to Theorem\ref{thm:perspective_derivative}, we get the following alternate expansions of partial derivatives of Tutte polynomials of matroid perspectives.

\medskip
\begin{corollary}
\label{cor:p_alternate}
We have
$${{\partial^{p+q} t}\over{{\partial x^p}{\partial y^q}}}                                         
(M,M';x,y,z)
=p!q!\sum_{\buildrel {A\subseteq E}\over
{\buildrel\scriptstyle {\iota_{M'}(A)=p}\over{\scriptstyle nl_M(A)=q}}}
x^{cr_{M'}(A)}y^{\epsilon_M(A)}z^{rcd_{M,M'}(A))}$$
$${{\partial^{p+q} t}\over{{\partial x^p}{\partial y^q}}}                                         
(M,M';x,y,z)
=p!q!\sum_{\buildrel {A\subseteq E}\over
{\buildrel\scriptstyle {cr_{M'}(A)=p}\over{\scriptstyle \epsilon_M(A)=q}}}
x^{\iota_{M'}(A)}y^{nl_M(A)}z^{rcd_{M,M'}(A))}$$

\hspace{0.5cm}
$\displaystyle{{\partial^{p+q} t}\over{{\partial x^p}{\partial y^q}}}                                         
(M,M';x,y,z)
=p!q!\sum_{\buildrel {A\subseteq E}\over
{\buildrel\scriptstyle {\iota_{M'}(A)=p}\over{\scriptstyle \epsilon_M(A)=q}}}
x^{cr_{M'}(A)}y^{nl_M(A)}z^{rcd_{M,M'}(A))}$
\hfill\square
\end{corollary}

\medskip
It follows from Theorem \ref{thm:dualities} that we get the same summands in expansions of Tutte polynomials
by exchanging the values of $x$ and $u$, and/or the values of $y$ and $v$,
in the symbols $[[x,u,y,v]]$ of Sections 1 and 2.
The example of Table 2 shows that expansions not related by these involutions may have different summands.

\medskip
Note that $\phi_{M,M'}(A)$ and $\phi^*_{M,M'}(A)$ define the same Dawson interval as $A$.
The composite involution $\phi_{M,M'}\circ\phi^*_{M,M'}$ is the lattice complementation 
in the Dawson interval.

\bigskip

\bigskip\bigskip
\noindent
{\sectionbf Acknowledgements}

\medskip\noindent
A first version of the paper was accepted by J. Combinatorial Theory series B in 2002.
However, during the process of correcting the galley proofs the author became aware of overlaps with a paper of G. Gordon and L. Traldi \cite{GT90}.
The paper was withdrawn for revision, a revision long postponed until very recently.
Meanwhile, the author happened to know of the paper of J.E. Dawson \cite{Da81}.
We mention that Theorem \ref{thm:matroid_partial_derivative} in the case of graphs appears as Theorem 27, with a sketch of proof, in the 2011 survey of J.A. Ellis-Monaghan and C. Merino \cite{EMM11} on Tutte polynomials of graphs.

The author gratefully thanks Steven Noble, James Oxley, Dominic Welsh and Thomas Zaslavsky for kindly answering his questions about the state of the art on Tutte polynomial derivatives, and also, Henry H. Crapo, Jeremy H. Dawson, and specially Lorenzo Traldi, for comments about the present paper.

\smallskip
\noindent
Michel Las Vergnas\\
Universit\'e Pierre et Marie Curie (Paris 6)\\
case 247 - Institut de Math\'ematiques de Jussieu\\
Combinatoire \& Optimisation\\ 
4 place Jussieu, 75252 Paris cedex 05 (France)\\

\noindent
{\it mlv@math.jussieu.fr}

\end{document}